\documentclass[12pt, a4paper, reqno]{amsart}

\usepackage[utf8]{inputenc}
\usepackage{amsfonts,amssymb,amsxtra,amsthm,amsmath,amscd,mathrsfs,cite}
\usepackage[all]{xy}
\usepackage{enumerate}  
\usepackage{tikz}
\usepackage[normalem]{ulem}
\usepackage{soul}
\usepackage{color}

\newtheorem{theorem}{Theorem}[section]
\newtheorem{lemma}[theorem]{Lemma}
\newtheorem{proposition}[theorem]{Proposition}
	\theoremstyle{remark}
	
	\numberwithin{equation}{section}
	\addtocounter{footnote}{1}

	\title[Furstenberg's flow on $\mathbb{T}^\omega$ in short intervals]  
	{M\"obius disjointness conjecture \\ for Furstenberg's flow on $\mathbb{T}^\omega$ in short intervals} 
	\author{Shuyang He,  Qingyang Liu,  Jing Ma}
	\address{School of Mathematics, Jilin University, Changchun, 130012, People's Republic of China}
	\email{syhe24@mails.jlu.edu.cn}
	\address{School of Mathematics, Renmin University of China, Beijing, 100872, People's Republic of China}
	\email{qingyangliu@ruc.edu.cn}
	\address{School of Mathematics, Jilin University, Changchun, 130012, People's Republic of China}
	\email{jma@jlu.edu.cn}

	\keywords{M\"{o}bius Disjointness Conjecture, irregular flows, skew product, infinite-dimensional torus. }
	\subjclass[2020]{11N37, 11L03.}

\begin{document}
		\begin{abstract}
			Furstenberg's flow on the infinite-dimensional torus $\mathbb{T}^\omega$ is defined by 
			\[			
			T (x_1, x_2, \ldots, x_\nu, \ldots)  
			= (x_1 + \alpha, x_2 + h(x_1),  \ldots, x_\nu + h(x_1 + (\nu-2)\beta), \ldots)
			\] 
			with $\alpha\in \mathbb{R}$ satisfying certain Diophantine conditions, 
			$\beta\in \mathbb{R}\backslash\mathbb{Q},$ and 
			$h: \mathbb{R}\to \mathbb{R}$ being $1$-periodic and analytic.  
			This flow  
			is irregular in the sense that its Birkhoff average does not exist for some $x\in \mathbb{T}^\omega$, and it 
			is a generalization of Furstenberg's irregular flow on $\mathbb{T}^2$. 
			The main result of this paper is that the M\"{o}bius Disjointness Conjecture of Sarnak 
			holds for the above flow $(\mathbb{T}^\omega, T)$ in short intervals $(N-M, N]$ with 
			$N^{5/8+\varepsilon} \leqslant M\leqslant N$.  
		\end{abstract}
		
		\maketitle
		
		\section{Introduction}
		\subsection{The M\"{o}bius Disjointness Conjecture of Sarnak} 	
		
		Let $\mu(n)$ denote the M\"{o}bius function, defined by $\mu(1)=1$, $\mu(n)=(-1)^k$ if $n$ is a product of 
		$k$ distinct primes, and $\mu(n)=0$ otherwise. 
		The behavior of $\mu(n)$ is of fundamental importance in analytic number theory, 
		particularly in the study of prime numbers. 
		A flow is a pair $(X, T)$, where $X$ is a compact topological space and 
		$T \colon X \to X$ a continuous map.  
		Sarnak's M\"{o}bius Disjointness Conjecture \cite{Sar12}, MDC in brief, states the following.
		
		\medskip
		
		\noindent
		\textbf{M\"{o}bius Disjointness Conjecture (Sarnak).}  
		\emph{The M\"{o}bius function $\mu$ is linearly disjoint from every zero-entropy flow $(X,T)$.  
			Equivalently, for any $f \in C(X)$ and any $x \in X$,
			\begin{equation}\label{MDC}
				\lim_{N\to\infty} \frac1N \sum_{n \leqslant N} \mu(n) \, f(T^n x) = 0 .
		\end{equation}}
		
		\medskip
		
		This conjecture establishes a profound connection between number theory and dynamical systems theory, 
		stimulating extensive research at the interface of analytic number theory and ergodic theory. 
		The simplest instance occurs when $X$ is a single point set and $T$ the identity map; this case is equivalent 
		to the Prime Number Theorem.  
		The first nontrivial case, where $X = \mathbb{T}$ is the unit circle and $T$ a translation, 
		corresponds to Vinogradov's estimate for exponential sums over primes, which played a crucial role in 
		his proof of the ternary Goldbach problem.

		Over the years, the conjecture has been established for a number of flows, 
		see for example \cite{Bou13, BouSarZie13, Sar12, LiuSar15, GreTao12, FerKulLem18} and the references therein, and 
		most of these flows are regular, which means that the limit of their ergodic averages
		\begin{eqnarray}\label{def/Birkh}
			\lim_{N\rightarrow\infty}\frac{1}{N}\sum_{n\leqslant  N} f(T^n x)
		\end{eqnarray}
		exist for every $f\in C(X)$ and every $x\in X$. 
		Furstenberg \cite{Fur61} 
		studied nonlinear smooth skew products $({\Bbb T}^2, T)$
		with ${\Bbb T}^2$ being the $2$-torus, and $T: {\Bbb T}^2\to {\Bbb T}^2$ given by
		\begin{eqnarray}\label{def/SKEW}
			T(x_1, x_2)= (x_1+\alpha, x_2+ h(x_1)),
		\end{eqnarray}
		where $\alpha\in \Bbb R$ and 
		$h: \mathbb{R}\to \mathbb{R}$ is a smooth $1$-periodic function. 
		As is common in the Kolmogorov-Arnold-Moser theory,
		the behavior of these
		flows depends on, among other things, the Diophantine property of the real number $\alpha$. 
		When $\alpha$ is irrational and close to rational numbers with small denominators, 
		these flows $({\Bbb T}^2, T)$ may be irregular,
		i.e. their ergodic averages in (\ref{def/Birkh}) may fail to exist.
		Furstenberg \cite{Fur61} gave examples of this kind, which are named {\it Furstenberg's irregular flows}
		in this paper, and which will be described in detail in \S \ref{Firrflow}.
		
		Liu and Sarnak \cite{LiuSar17} defined a wider class of irregular smooth skew products on ${\Bbb T}^2$, and established MDC for these irregular flows. The construction in  \cite{LiuSar17} depends on the continued fraction expansion of $\alpha$ as well as its convergents $l_k/q_k$ with $k=1, 2, \ldots$; for detailed properties of continued fractions, see Lemma~\ref{con/fra/lem} in this 
		paper.  
		Suppose that these denominators of convengents of $\alpha$ satisfy 
		\begin{eqnarray}\label{qm+1>eqm}
			q_{k+1}\asymp e^{q_k}
		\end{eqnarray}
		for all positive $k$. Define $q_{-k}= - q_k$, and put  
		\begin{eqnarray}\label{def/hx=sum}
			h(x)=\sum_{k\not=0} t_k (1-e(q_k\alpha))e(q_k x), 
		\end{eqnarray}
		where the coefficients
		$t_k$ satisfy
		\begin{eqnarray}\label{|c|<C}
			t_k= t_{-k}, \quad |t_k|\leqslant  \tau
		\end{eqnarray}
		for all positive $k$ and for some positive constant $\tau\geqslant 1$. 
		We note that if $t_k$ is not bounded then $g(x)$ may not be $L^2$. Also, by (\ref{|c|<C})
		and $q_{-k}=-q_k$, we have $\overline{h(x)}=h(x)$ which means $h$ is real.
		It is proved in \cite{LiuSar17} that, 
		if 	$T$ is 
		given by (\ref{def/SKEW}), $\alpha$ by (\ref{qm+1>eqm}), and $h$ by (\ref{def/hx=sum}) and (\ref{|c|<C}), then $({\Bbb T}^2, T)$ is irregular, yet its MDC holds with a rate.
		
		\subsection{Furstenberg's irregular flow on infinite dimensional torus}\label{Firrflow}    	 
		The second named author \cite{QLiu19} defined a class of skew products  
		on the infinite dimensional torus ${\Bbb T}^\omega$, 
		where ${\Bbb T}^\omega = \prod_{\Bbb N} {\Bbb T}$ is the direct product 
		of countably many copies of $\Bbb T$. Define $T: {\Bbb T}^\omega\to {\Bbb T}^\omega$
		by
		\begin{align}\label{def/Tome}
			T(x_1, x_2, \ldots, x_k, \ldots)
			=(x_1+\alpha, x_2+h(x_1), \ldots, x_k+h(x_1+ (k-2)\beta), \ldots)
		\end{align}
		with $\alpha$ satisfying (\ref{qm+1>eqm}), $h$ satisfying (\ref{def/hx=sum}) and (\ref{|c|<C}),
		and $\beta$ being an irrational number.
		
		These flows $({\Bbb T}^\omega, T)$ generalize the earlier irregular flows on ${\Bbb T}^2$  
		and are referred to as {\it Furstenberg's irregular flows on} ${\Bbb T}^\omega$.  
		Under certain technical hypotheses, Auslander \cite{Aus88} established that such flows are distal.
		In \S\ref{Tdistal} we relax these hypotheses \eqref{qm+1>eqm}, 
		\eqref{def/hx=sum} and \eqref{|c|<C}, and prove that every system of the form \eqref{def/Tome} with a 
		$1$-periodic function $h:\mathbb{R}\rightarrow\mathbb{R}$ is distal.
		Consequently, each has zero topological entropy, and so Sarnak’s 
		MDC is expected to hold for them.
		
		In \cite{QLiu19}, MDC on $({\Bbb T}^\omega, T)$ 
		with $h$ satisfying \eqref{qm+1>eqm}, \eqref{def/hx=sum}  and \eqref{|c|<C} is established.  
		We remark that the second named author \cite{QLiu19} actually worked on the 
		flow $({\Bbb T}^\omega, Q)$ with 
		\begin{align*}
			Q(x_1, x_2, \ldots, x_k, \ldots)
			=(x_1+\alpha, x_2+h(x_1), \ldots, x_k+h(x_1+ \beta^{k-2}), \ldots), 
		\end{align*}
		but these carry plainly over to $({\Bbb T}^\omega, T)$ under consideration at present.   
		
		\subsection{Main results of this paper}  
		In the rest of the paper, we will widen the meaning of Furstenberg's irregular flows $({\Bbb T}^\omega, T)$ 
		by putting less restrictive condition on $h$. Roughly speaking, the previous restrictions \eqref{def/hx=sum} and \eqref{|c|<C} 
		are no longer needed. As we will show in Section \ref{Tdistal} that the skew product we discuss 
		in this paper is also distal. The purpose of this paper is to establish MDC for 
		Furstenberg's irregular flows $({\Bbb T}^\omega, T)$
		in short  intervals. Our main results are as follows.
		
		\begin{theorem}\label{thm1}
			Let $h: \mathbb{R} \to \mathbb{R}$ be $1$-periodic and analytic.   
			Let both $\alpha$ and $\beta$ be irrational, 
			and let the denominators $\{q_k\}_{k=0}^\infty$ of the convergents of $\alpha$ satisfy  
			$q_{k+1} \asymp e^{q_k}$.  
			Let $(\mathbb{T}^\omega,T)$ be defined as in \eqref{def/Tome}.   
			Let $\varepsilon > 0$ be arbitrarily small and  
			\begin{align}\label{range/M}  
				N^{5/8+\varepsilon} \leqslant M \leqslant N.  
			\end{align}  
			Then, for any $f \in C(\mathbb{T}^\omega)$ and any $x \in \mathbb{T}^\omega$,  
			\begin{align}\label{M/wo/Rate}  
				\sum_{N-M < n \leqslant N} \mu(n) f(T^n(x)) = o(M).  
			\end{align}
		\end{theorem}

		
		Applying 
		Diophantine approximation theory 
		and Fourier analysis, 
		we obtain another version of the  M\"{o}bius Disjointness Conjecture along short intervals for the skew products defined 
		in \eqref{def/Tome} with $h$ therein relaxed to a $\tau$-smooth function for some fixed $\tau>3$, which is a short interval version of the main theorem of Wang-Yao \cite{WanYao20}. 
		
		\begin{theorem}\label{thm2} 
			Let $h: \mathbb{R} \to \mathbb{R}$ be $1$-periodic and $\tau$-smooth for some fixed $\tau>3$. 
			Let both $\alpha$ and $\beta$ be irrational, 
			and let the denominators $\{q_k\}_{k=0}^\infty$ of the convergents of $\alpha$ satisfy  
			$q_{k+1}\ll q_k^\tau$.   
			Let $(\mathbb{T}^\omega,T)$ be defined as in \eqref{def/Tome}.   
			Then for $M$ as in \eqref{range/M}, we have 
			\eqref{M/wo/Rate} for any $f \in C(\mathbb{T}^\omega)$ and any $x \in \mathbb{T}^\omega$.  
		\end{theorem}

		For completeness, using Zhan's result \cite[Theorem 5]{Zha91} on exponential sum over primes in short intervals, we obtain  a short term version of the  MDC 
		for the skew products defined in \eqref{def/Tome}  with rational $\alpha$ and $1$-periodic function $h$.
		
		\begin{theorem}\label{thm3} 
			Let $(\mathbb{T}^\omega,T)$ be defined as in \eqref{def/Tome} with $h: \mathbb{R} \to \mathbb{R}$ be $1$-periodic,  $\alpha=l/q$ with $(l, q)=1$ and  $\beta$ irrational. 
			Then for $M$ as in \eqref{range/M}, we have 
			$$
			\sum_{N-M < n \leqslant N} \mu(n) f(T^n(x)) 
			\ll_{A,q,\varepsilon} \frac{M}{\log^A N}  
			$$
			for any $f \in C(\mathbb{T}^\omega)$ and any $x \in \mathbb{T}^\omega$.  
		\end{theorem}
		
		From these three theorems we see that MDC hold on the flow $(\mathbb{T}^\omega,T)$ defined in \eqref{def/Tome} with $\alpha$ rational or irrational with the denominators $\{q_k\}_{k=0}^\infty$ of the convergents of $\alpha$ satisfying  $q_{k+1} \asymp e^{q_k}$ or $q_{k+1}\ll q_k^\tau$.   
		We want to emphasize here that the point of this paper is how short the interval $(N-M, N]$ is, not how general $\alpha$ can be. 
		The $\alpha$ in this paper is designed (1) to carry the irregularity of the flow $(\mathbb{T}^\omega, T)$, and (2) to shorten the interval. 
		In \cite{LiuMaWang} we have also proved MDC for the flow $(\mathbb{T}^\omega,T)$ 
		defined in \eqref{def/Tome} with $\alpha\in \mathbb{R}$ and 
		$h: \mathbb{R}\to \mathbb{R}$ being $1$-periodic and $C^{1+\varepsilon}$-smooth.

		\subsection{General setting}   
		Liu and Sarnak \cite{LiuSar15} have established the MDC for
		nonlinear smooth skew products $({\Bbb T}^2, T)$ with $T$ given by (\ref{def/SKEW}),
		that is for all $\alpha\in \Bbb R$ and for all $h$ satisfying some minor technical hypothesis such that the flow is irregular.  
		In particular, MDC holds for Furstenberg's irregular flows, that is (\ref{MDC}) is established
		even when (\ref{def/Birkh}) does not exist.  Liu and Sarnak \cite{LiuSar17} 
		have also treated Furstenberg's irregular flows directly, and established the MDC 
		with a rate.
		The technical hypothesis on $h$ in
		\cite{LiuSar15} is removed by Wang \cite{Wan17}.
		Subsequently, the continuity requirement on $h$ was further relaxed: 
		Huang--Wang--Ye \cite{HuaWanYe19} proved 
		MDC for $({\Bbb T}^2, T)$ with $h$ being $C^\infty$-smooth, 
		Kanigowski--Lemanczyk--Radziwill \cite{KanLamRad21} for 
		$h$ of $C^{2+\varepsilon}$, and de Faveri \cite{Fav22} for $h$ of $C^{1+\varepsilon}$.

		In 2021, Huang--Liu--Wang \cite{HuaLiuWan21} considered a class of skew products on $\mathbb{T}\times \Gamma\backslash G$, where  $\mathbb{T}$ is still the unit circle, $\Gamma\backslash G$ is the $3$-dimensional Heisenberg nilmanifold. 
		Explicitly, let
		$$
		G=\begin{pmatrix}\begin{smallmatrix} 1 &\mathbb{R}& \mathbb{R}\\ & 1 & \mathbb{R}\\ & & 1\end{smallmatrix}\end{pmatrix} 
		\quad \mbox{ and } \quad 
		\Gamma =\begin{pmatrix}\begin{smallmatrix} 1 & \mathbb{Z}& \mathbb{Z}\\ & 1 & \mathbb{Z} \\ & & 1\end{smallmatrix}\end{pmatrix}.
		$$
		Huang--Liu--Wang \cite{HuaLiuWan21} defined a skew product on $\mathbb{T}\times\Gamma\backslash G$ by the map  $S_\alpha:\mathbb{T}\times\Gamma\backslash G \to \mathbb{T}\times\Gamma\backslash G$,  
		$$
		S_\alpha :(t, \Gamma g)\mapsto
		\left(
		t +\alpha, \Gamma g
		\begin{pmatrix}
			\begin{smallmatrix}
				1 & \varphi(t) & \psi(t) \\
				0 & 1 & \eta(t) \\
				0 & 0 & 1    
			\end{smallmatrix}
		\end{pmatrix}
		\right),
		$$
		where $\varphi$, $\eta$ and $\psi$ are  periodic functions on $\mathbb{R}$ with period $1$, and $\alpha\in [0,1)$. Assume  $\varphi$, $\eta$ and $\psi$ are $C^{\infty}$-smooth.  
		Huang-Liu-Wang \cite{HuaLiuWan21} showed that the flow $(\mathbb{T}\times\Gamma\backslash G, S_\alpha)$ is distal, and the MDC holds for $(\mathbb{T}\times\Gamma\backslash G, S_\alpha)$  if  $\varphi = \eta$  have zero mean.
		He-Wang \cite{HeWang23} weaken the continuity requirement on $\varphi = \eta$ and $\psi$ to  $C^{1+\varepsilon}$. 
		Ma and Wu \cite{MaWu24} relaxed the condition  $\varphi = \eta$ but kept the smoothness assumption, the M\"obius disjointness of $(\mathbb{T}\times\Gamma\backslash G, S_{\alpha})$ was derived in the case of rational $\alpha$ without the constraint $\eta=\varphi$ 
		and in the case of irrational $\alpha$ with the condition  $\eta=k\varphi$, where $k$ is a fixed nonzero real number. 
		Lau and Ma \cite{LauMa} derived the M\"obius disjointness of $(\mathbb{T}\times\Gamma\backslash G, S_{\alpha})$ for irrational $\alpha$ with no restriction on  $\eta$ or $\varphi$ and weaken the continuity requirement on $\varphi,\eta$ and $\psi$ to  $C^{1+\varepsilon}$ and $C^{2+2\varepsilon}$ respectively.
		
		\medskip 
		\noindent
		\textbf{Notation.}
		For $x\in\mathbb{R}$, 
		$\|x\|$   denotes the distance from $x$ to the nearest integer, 
		that is, $\|x\|=\inf_{n\in\mathbb{Z}}|x-n|$. If we have $A\ll B$ as well as $B\ll A$, 
		we will write $A\asymp B$.  
		We will write $e(x)$ for $e^{2\pi ix}$ and $\exp(x)$ for $e^x$. 
		The well-known fact that $|e(x)-1|\asymp \|x\|$ will be repeatedly used in this paper.
		
		\section{Preliminaries}\label{Sec3}
		In this section, we will state some auxiliary results we need in the proof. 
		
		Our argument will depend on the continued fraction expansion of $\alpha$. 
		Every real number $\alpha$ has its continued fraction representation
		\begin{align}
			\label{continuedfraction}
			\alpha=a_0+\frac{1}{a_1+\frac{1}{a_2+\cdots}},
		\end{align}
		where $a_0=[\alpha]$ is the integral part of $\alpha$, 
		and $a_1, a_2, \ldots$ are positive integers. 
		The expression \eqref{continuedfraction} is infinite if $\alpha\notin\mathbb{Q}$. 
		The right-hand side of \eqref{continuedfraction} can be written in the form $[a_0;a_1,a_2,\ldots]$, 
		which is the limit of finite continued expressions
		\begin{align*}
			[a_0;a_1,a_2,\ldots,a_k]=a_0+\frac{1}{a_1+\frac{1}{a_2+\frac{1}{\ddots+\frac{1}{a_k}}}}.
		\end{align*}
		Let $l_k/q_k=[a_0;a_1,a_2,\ldots,a_k]$ be the $k$-th convergent  of $\alpha$.  
		Some well-known properties of $l_k/q_k$ are summarized in the following lemma.

		\begin{lemma}
			\label{con/fra/lem}
			Let $\alpha\in [0,1)$ be irrational and  $l_k/q_k$ be the $k$-th convergent of $\alpha$. Then, 
			
			{\rm (1)} we have 
			$q_0 = 1, q_1 = a_1$ and $q_{k+2} = a_{k+2}q_{k+1} + q_k$ for all $k \geqslant 0$;  
			
			{\rm (2)} we have, for any $k \geqslant 1$, 
			\begin{equation*}
				\frac{1}{2q_{k+1}} < \|q_k\alpha\| < \frac{1}{q_{k+1}}; 
			\end{equation*}   
			
			{\rm (3)} 
			if $|\alpha - l/q| < 1/(2q^2)$ for some $l, q\in\mathbb{Z}$,  
			then $l/q = l_k/q_k$ for some $k \geqslant 1$;
			
			{\rm (4)}
			we have  $(l_k,q_k)=1$ for any $k\geqslant1$.
		\end{lemma}

		Wang \cite[Lemma 3.2]{KWang22}		
		proved the following estimates for the Fourier coefficients of a class of exponential functions.
		
		\begin{lemma}\label{bd/coef/W}
			Let $f\in C(\mathbb{R})$ be $1$-periodic 
			and $F(x)=e(f(x))$ have the Fourier expansion 
			\[
			F(x)=\sum_{m\in\mathbb{Z}}c(m)e(mx).
			\] 
			Then we have 
			\[
			c(m)\ll \frac{1}{m^2}(\|f'\|_{\infty}^2+\|f''\|_\infty)
			\]
			for any $m\neq 0$, 
			where $\|\cdot\|_\infty$ denotes the uniform norm.
		\end{lemma}
		
		The following result of Zhan \cite[Theorem 5]{Zha91}  on 
		exponential sum over primes in short intervals is necessary in our proofs. 
		In \cite[Theorem 5]{Zha91} only the case $q=r=1$ is treated, but it is plain to 
		see that his proof carries over to the general case below. 
		
		\begin{lemma}\label{zhan}
			Let $q$ and $r\leqslant  q$ be fixed co-prime positive integers. 
			Let $A>0$ and $\varepsilon>0$ be arbitrarily large and small respectively, 
			and assume $N^{5/8+\varepsilon}\leqslant M\leqslant N$. Then 
			\[
			\sum_{\substack{N-M<n\leqslant N\\ n\equiv r \bmod q}} 
			\mu(n)e(\alpha n)\ll_{A,\varepsilon, q} \frac{M}{\log^A N}  
			\]
			holds uniformly for $\alpha\in \mathbb{R}$.  
		\end{lemma}
		
		
		\section{The skew products $(\mathbb{T}^\omega,T)$ are distal} \label{Tdistal}  
		In this section we will show that the skew products defined in \eqref{def/Tome} are distal. 
		
		\begin{proposition}\label{distal}
			Let $T:\mathbb{T}^\omega \to \mathbb{T}^\omega $ be of the form  \eqref{def/Tome}  
			with $\alpha, \beta \in \mathbb{R}$ and $h: \mathbb{R}\to \mathbb{R}$ being $1$-periodic, not necessarily continuous. 
			Then the flow $(\mathbb{T}^\omega,T)$ is distal.
		\end{proposition}
		
		\begin{proof}    
			For any $x, y\in \mathbb{T}^\omega $, write $x=(x_1, \ldots, x_\nu,\ldots)$ and 
			$y=(y_1, \ldots, y_\nu,\ldots)$, and 
			define the distance between $x$ and $y$ 
			by 
			\[
			d(x, y):=\sum_{\nu=1}^{\infty}\frac{1}{2^\nu}\| x_\nu-y_\nu\|.
			\]
			Applying \eqref{def/Tome} repeatedly, we get 
			\begin{equation}\label{Tnx=}
				T^n (x_1,\ldots, x_\nu,\ldots)
				= (x_1(n), \ldots, x_\nu(n),\ldots)
			\end{equation} 
			with $x_1(n)=x_1+n\alpha$ and, for all $\nu\geqslant 2$,  
			\begin{align}\label{Tnx=//2}
				x_\nu(n)=x_\nu+\sum_{j=0}^{n-1}h(x_1+j\alpha+(\nu-2)\beta). 
			\end{align} 
			
			Now we suppose that $x\neq y$, and we write similarly to \eqref{Tnx=} and \eqref{Tnx=//2} 
			that 
			\[
			T^n(y_1,\ldots, y_\nu,\ldots)=(y_1(n),\ldots,y_\nu(n),\ldots),
			\] 
			with $y_1(n)=y_1+n\alpha$ and, for all $\nu\geqslant 2$,  
			\begin{align*}
				y_\nu(n)=y_\nu+\sum_{j=0}^{n-1}h(y_1+j\alpha+(\nu-2)\beta). 
			\end{align*} 
			Let $\nu_0$ be the minimal subscript $\nu$ such that $x_\nu\neq y_\nu$. 
			If $\nu_0=1$, then 
			\begin{align}
				\inf_{n\geqslant0} d(T^n {x},T^n {y})
				\geqslant  \inf_{n\geqslant 0} \frac{1}{2}\|(x_1+n\alpha)-(y_1+n\alpha)\|
				\geqslant\frac{1}{2}\|x_1-y_1\|>0.\notag
			\end{align}
			If $\nu_0>1$, then $x_1=y_1$ and
			\begin{align*}
				&\inf_{n\geqslant 0} d(T^n x,T^n y)\\
				&=\inf_{n\geqslant 0} \sum_{\nu=2}^{\infty}\frac{1}{2^\nu}
				\bigg\|\bigg(x_\nu+\sum_{j=0}^{n-1}h (x_1+j\alpha+(\nu-2)\beta)\bigg)
				-\bigg(y_\nu+\sum_{j=0}^{n-1}h (y_1+j\alpha+(\nu-2)\beta)\bigg)\bigg\| \\ 
				&=\inf_{n\geqslant 0} \sum_{\nu=2}^{\infty}\frac{1}{2^\nu}
				\|x_\nu-y_\nu\| 
				\geqslant \frac{1}{2^{\nu_0}} \|x_{\nu_0}-y_{\nu_0}\| >0. 
			\end{align*}
			We conclude that $\inf_{n\geqslant0}  d(T^n {x},T^n {y})>0$ whenever $x\neq y$ 
			in $\mathbb{T}^\omega$. This proves that the flow $(\mathbb{T}^\omega,T)$  
			under consideration is indeed distal. 
		\end{proof}

		\section{Proof of Theorem \ref{thm1}}\label{sec4}
		\subsection{Analysis on $\mathbb{T}^\omega$}\label{Ana/TT}  
		We review some basic facts about Fourier analysis on the infinite-dimensional torus; for proofs, see
		e.g. Rudin \cite[\S 2.2.3-5]{Rud62}. 
		As stated before, the infinite-dimensional torus ${\Bbb T}^\omega$
		is defined by the direct product  
		${\Bbb T}^\omega = \prod_{\Bbb N} {\Bbb T}$ 
		of countably many copies of $\Bbb T$.  We know that $\Bbb T$ is a compact abelian group, and so is ${\Bbb T}^\omega$
		with the product topology. The dual group of ${\Bbb T}^\omega$ is the direct sum 
		${\Bbb Z}^\infty = \oplus_{\Bbb N} {\Bbb Z}$
		of countably many copies of $\Bbb Z$, where ${\Bbb Z}$ is the discrete abelian group of the integers.
		Each $x\in {\Bbb T}^\omega$ may be thought of as a string $x=(x_1, \ldots, x_k, \ldots)$,
		and each $\gamma\in {\Bbb Z}^\infty$ can be written as $(n_1, \ldots, n_k, \ldots)$, where only
		finitely many of the $n_k$'s are nonzero. The latter fact is important.
		
		For any $f\in L^1({\Bbb T}^\omega)$, its Fourier transform $\hat{f}$ is
		\begin{eqnarray*}\label{def/Hat/f}
			\hat{f}(n_1, \ldots, n_k, \ldots)
			=\int_{{\Bbb T}^\omega}f(x) e(-\langle x, \gamma\rangle) \textup{d} x
			=\int_{{\Bbb T}^\omega}f(x) e\bigg(-\sum_{k}x_k n_k\bigg)\textup{d}x,
		\end{eqnarray*}
		where only finitely many of the integers $n_k$'s are different from $0$, $\langle \cdot, \cdot\rangle$ means the dot product, and the $x_k$'s are real numbers modulo $1$.
		Thus, the inversion formula has the form
		\begin{eqnarray*}\label{def/Hat/f}
			f(x_1, \ldots, x_k, \ldots)=\sum \hat{f} (n_1, \ldots, n_k, \ldots) e\bigg(\sum_{k}x_k n_k\bigg).
		\end{eqnarray*}
		It is also known that the set of trigonometric functions on ${\Bbb T}^\omega$
		is dense in $C({\Bbb T}^\omega)$.
		
		\subsection{Proof of Theorem~\ref{thm1}: starting point}   
		
		Since the space of trigonometric functions is dense in $C(\mathbb{T}^\omega)$,  our first target 
		is to prove that, for any trigonometric functions $f$ on $\mathbb{T}^\omega$ and 
		any $x \in \mathbb{T}^\omega$, we have 
		\begin{align}\label{M/Rate}  
			\sum_{N-M < n \leqslant N} \mu(n) f(T^n(x)) \ll_A \frac{M}{\log^A N},   
		\end{align}  
		where $A > 0$ is arbitrary. One sees that the above estimate has a rate, and is therefore more stronger.

		By the structure of $\mathbb{Z}^\infty$ in \S\ref{Ana/TT}, 
		any $b\in \mathbb{Z}^\infty$ must be of the form 
		$b=(b_1, \ldots,  b_{\nu'}, 0, \ldots)$ for some $\nu'=\nu_b'$.  
		Denote
		$\rho(b):=\sum_{\nu=2}^{\nu'}b_\nu$  
		and $\|b\|:=\sum_{\nu=2}^{\nu'}|b_\nu|$. 
		To prove \eqref{M/Rate} for any trigonometric functions on $\mathbb{T}^\omega$, 
		we will show that for any given  
		$b\in\mathbb{Z}^\infty$ and for any $x\in\mathbb{T}^\omega$
		we have 
		\begin{align}\label{th1/goa}
			S(N, M) 
			:=\sum_{N-M< n\leqslant N}\mu(n) e (\langle b, T^n(x)\rangle)
			\ll \frac{M}{\log^A N}, 
		\end{align} 
		where $A>0$ is arbitrary.

		Applying \eqref{def/Tome} repeatedly and using \eqref{Tnx=//2}, we get
		\[
		\langle b,T^n(x)\rangle=b_1 n\alpha +\sum_{\nu=1}^{\nu'}b_\nu x_\nu+\sum_{\nu=2}^{\nu'}b_{\nu}\sum_{j=0}^{n-1}h(x_1+j\alpha+(\nu-2)\beta).
		\]
		It follows that 
		\begin{align}\label{smn}
			S(M,N)
			&=e\bigg(\sum_{\nu=1}^{\nu'}b_\nu x_\nu\bigg) 
			\sum_{N-M<n\leqslant N}\mu(n) 
			e\bigg\{b_1n\alpha+\sum_{\nu=2}^{\nu'}b_\nu\sum_{j=0}^{n-1}h(x_1+j\alpha+(\nu-2)\beta)\bigg\}\notag
			\\
			&\ll\sum_{N-M<n\leqslant N}\mu(n)e\bigg\{b_1n\alpha+\sum_{\nu=2}^{\nu'}b_\nu\sum_{j=0}^{n-1}h(x_1+j\alpha+(\nu-2)\beta)\bigg\}.
		\end{align}  
		
		We will write $h$ as the sum of the resonant and the non-resonant parts. 
		For $k\geqslant 1$, define
		\[
		M_k:=\{m\in\mathbb{Z}: q_k\leqslant|m|<q_{k+1}\}, 
		\]
		and set $M_0:=\{0\}$ for convenience. 
		Then obviously $\mathbb{Z}=\bigcup_{k \geqslant 0} M_k$.  
		For $k\geqslant2$, define
		\[
		M_{k}^{\sharp} := \{ m \in M_k : q_k|m \}
		\]
		and 
		\[
		\mathbb{Z}^{\sharp}
		:=\bigcup_{k\geqslant2}M_{k}^{\sharp}\cup\{0\}, \qquad\qquad
		\mathbb{Z}^{\flat} :=\mathbb{Z}\setminus \mathbb{Z}^{\sharp}.	
		\] 
		Then we can decompose $h$ into 
		\begin{equation}
			\label{h12}
			h(t)
			=
			h_{1}(t)+h_{2}(t)
			:=\sum_{m\in\mathbb{Z}^\sharp}\hat{h}(m)e(mt)
			+\sum_{m\in \mathbb{Z}^{\flat}}\hat{h}(m)e(mt), 
		\end{equation}
		where $\hat{h}(m)$, $m\in\mathbb{Z}$,  are the Fourier coefficients of $h$. 
		We call $h_1$ and $h_2$ the resonant part and the non-resonant part of $h$ respectively. 
		
		As is pointed out by Liu and Sarnak in \cite{LiuSar15}, 
		the non-resonant part can be written as 
		the difference of some continuous $1$-periodic function $g$, 
		whose contribution can be eliminated by Fourier analysis. 
		Wang \cite[Lemma 3.1]{KWang22} present this as the following lemma.
		
		\begin{lemma}\label{WLS}
			Let $h:\mathbb{R}\rightarrow\mathbb{R}$ be a $1$-periodic analytic function with Fourier coefficients $\hat{h}(m)$ for $m\in\mathbb{Z}$. 
			Then, the series
			\[
			\sum_{m\in \mathbb{Z}^{\flat}}\hat{h}(m)\frac{e(mt)}{e(m\alpha)-1}
			\]
			converges uniformly to  an analytic function $g(t)$			of period $1$ such that
			\[
			h_{2}(t)=g(t+\alpha)-g(t)
			\]
			for all $t\in\mathbb{T}$. 
		\end{lemma}

		\subsection{Proof of Theorem~\ref{thm1}: proof of \eqref{th1/goa}}
		Let $h=h_1+h_2$ be  as in  \eqref{h12}. 
		By Lemma \ref{WLS}, the last summation over $j$ in \eqref{smn} can be written as
		\begin{align*}
			&  \sum_{j=0}^{n-1}h(x_1+j\alpha+(\nu-2)\beta) \\ 
			&  =\sum_{j=0}^{n-1}
			\{h_1(x_1+j\alpha+(\nu-2)\beta)+h_2(x_1+j\alpha+(\nu-2)\beta)\} \\
			&  =\sum_{j=0}^{n-1}\{h_1(x_1+j\alpha+(\nu-2)\beta)+g(x_1+(j+1)\alpha+(\nu-2)\beta)  
			-g(x_1+j\alpha+(\nu-2)\beta) \},
		\end{align*}
		where $g$ is defined in Lemma \ref{WLS}.
		The last line is 
		\[
		=g(x_1+n\alpha+(\nu-2)\beta)-g(x_1+(\nu-2)\beta)+\sum_{j=0}^{n-1}h_1(x_1+j\alpha+(\nu-2)\beta), 
		\]
		where we note that the second term does not involve $n$. Substituting this into 
		\eqref{smn}, we get
		\begin{align*}
			S(M,N)
			& \ll
			\sum_{N-M < n \leqslant N} 
			\mu(n) \, 
			e\bigg\{
			b_1 n \alpha 
			+ 
			\sum_{\nu=2}^{\nu'} b_\nu \bigg( 
			g(x_1 + n \alpha + (\nu-2)\beta) 
			\\
			& \quad +
			\sum_{j=0}^{n-1} h_1(x_1 + j \alpha + (\nu-2)\beta)
			\bigg)
			\bigg\}.
		\end{align*}
		
		Expanding $h_1$ into the Fourier series as in \eqref{h12}, we obtain
		\begin{align*}
			&\sum_{j=0}^{n-1}h_1(x_1+j\alpha+(\nu-2)\beta) \\ 
			& \quad =\sum_{m\in \mathbb{Z}^\sharp}\hat{h}(m)e(m x_1+m(\nu-2)\beta)\sum^{n-1}_{j=0}e(mj\alpha) 
			\\
			& \quad =\sum_{k\geqslant 2}\sum_{\substack{q_k\leqslant |m|<q_{k+1}\\ q_k|m}}\hat{h}(m)e(m x_1+m(\nu-2)\beta)\frac{e(mn\alpha)-1}{e(m\alpha)-1}+n\hat{h}(0).
		\end{align*}
		We will split the above sum at $k=K$ for some proper $K$.  
		Let $K$ be the positive integer satisfying 
		\begin{equation}\label{KlogN}
			q_K\leqslant 2\log N <q_{K+1}.	
		\end{equation} 
		Then, we get
		\begin{align*}
			& \sum_{j=0}^{n-1}h_1(x_1+j\alpha+(\nu-2)\beta)\notag\\
			& \quad =\sum_{2\leqslant k \leqslant K}\sum_{\substack{q_k\leqslant|m|<q_{k+1}\\q_k|m}}\hat{h}(m)e(m x_1+m(\nu-2)\beta)\frac{e(mn\alpha)-1}{e(m\alpha)-1}+n\hat{h}(0)+S_{\nu,K}(n)
		\end{align*}
		with
		\begin{align*}
			S_{\nu,K}(n)
			& :=\sum_{k> K}\sum_{\substack{q_k\leqslant |m|<q_{k+1}\\ q_k|m}}\hat{h}(m)e(m x_1+m(\nu-2)\beta) \frac{e(mn\alpha)-1}{e(m\alpha)-1} \\ 
			& \ll N\sum_{|m|>2\log N}e^{-\eta |m|}\ll N^{-1}
		\end{align*} 
		by noting that the analyticity of $h$ implies 
		$
		\hat{h}(m)\ll e^{-\eta |m|}
		$
		where $\eta>0$ is some constant.  
		For $2\leqslant k\leqslant K$, denote
		\[
		f_{\nu,k}(t):=\sum_{\substack{q_k\leqslant |m|<q_{k+1}\\q_k|m}}\hat{h}(m)e(mx_1+m(\nu-2)\beta)\frac{e(mt)-1}{e(m\alpha)-1}.
		\]
		Then, using $e(\sum_{\nu=2}^{\nu'}b_\nu S_{\nu,K}(n))=1+O(\|b\|N^{-1})$, we get
		\begin{align}\label{smn1}
			S(M,N)\ll S_1(M,N)+ \|b\|,
		\end{align}
		where
		\begin{align*}
			S_1(M,N)
			& := 
			\sum_{N-M < n \leqslant N} \mu(n) \, e\bigg\{b_1 n \alpha
			+\sum_{\nu=2}^{\nu'} 
			b_\nu \bigg(
			g(x_1 + n \alpha + (\nu-2)\beta) \\
			& \quad + \sum_{2\leqslant k\leqslant K} f_{\nu,k}(n\alpha)+n\hat{h}(0)
			\bigg)\bigg\}.
		\end{align*}
		So, to obtain \eqref{th1/goa}, we first need to estimate  $S_1(M,N)$.
		
		Let
		\[
		G(t)
		=e\bigg\{\sum_{\nu=2}^{\nu'}b_\nu 
		g(x_1 + t + (\nu-2)\beta)\bigg\}, 
		\] 
		and, for $k=2, \dots, K$, 
		\[
		F_{k}(t)=e\bigg\{\sum_{\nu=2}^{\nu'}b_\nu f_{\nu,k}(t)\bigg\}. 
		\]
		Then $G(t)$ and $F_{k}(t)$ with $k=2, \dots, K$ are $1$-periodic  analytic functions
		which can be expanded into Fourier series
		\[
		G(t)=\sum_{l\in\mathbb{Z}}a(l)e(lt), 
		\qquad 
		F_{k}(t)=\sum_{l\in\mathbb{Z}}b^{(k)}{(l)}e(l t). 
		\]
		Substituting these into $S_1(M,N)$, we get 
		\begin{align*}
			S_1(M,N)
			&=\sum_{N-M<n\leqslant N}
			\mu(n)e\bigg\{b_1n\alpha
			+   \rho(b) n\hat{h}(0) \bigg\} 
			G(n\alpha)\prod _{k=2}^{K}F_{k}(n\alpha)\notag\\
			&=\sum_{l\in\mathbb{Z}}a(l)\sum_{l_2\in\mathbb{Z}}b^{(2)}{(l_2)}\cdots\sum_{l_K\in\mathbb{Z}}b^{(K)}{(l_K)}\notag\\
			& \quad \times 
			\sum_{N-M<n\leqslant N}\mu(n)e\Big\{n\alpha\big(b_1+l+l_2+\ldots+l_K\big)+ n\rho(b) \hat{h}(0) \Big\}. 
		\end{align*}
		By Lemma \ref{zhan}, the last sum over $n$ can be estimated as 
		\[
		\ll_{D,\varepsilon} \frac{M}{\log^{D} N}  
		\]
		for any $D>0$ and any $N^{5/8+\varepsilon}\leqslant M\leqslant N$. Also, 
		it is important that the above estimate is uniform in $l, l_2, \ldots, l_K$. Thus,  
		\begin{align}\label{s1mn}
			S_1(M,N)
			\ll_{D,\varepsilon}
			\frac{M}{\log^{D} N}
			\sum_{l\in\mathbb{Z}} |a(l)| \sum_{l_2\in\mathbb{Z}} |b^{(2)}(l_2) | \cdots \sum_{l_K\in\mathbb{Z}} | b^{(K)} (l_K)| 
		\end{align}
		for any $D>0$ and any $N^{5/8+\varepsilon}\leqslant M\leqslant N$. 		
		
		Now, we will use Lemma \ref{bd/coef/W} to estimate the summations over $l, l_2, \ldots, l_K$ in \eqref{s1mn}. 
		Firstly, we need to calculate the uniform norm of $G(t)$ and of $F_k(t)$ for $k=2, \ldots, K$. 
		For $G(t)$, by the definition of $g(t)$ in Lemma \ref{WLS}, 
		we get
		\begin{align}
			& \bigg(\sum_{\nu=2}^{\nu'}b_\nu 
			g(x_1 + t + (\nu-2)\beta)\bigg)'
			=2\pi i
			\sum_{\nu=2}^{\nu'}b_\nu 			\sum_{m\in\mathbb{Z}^{\flat}}m\hat{h}(m)\frac{e\big(m(x_1 + t + (\nu-2)\beta)\big)}{e(m\alpha)-1},\notag 
			\\
			&\bigg(\sum_{\nu=2}^{\nu'}b_\nu g(x_1 + t + (\nu-2)\beta)\bigg)''
			=-4\pi^2 \sum_{\nu=2}^{\nu'}b_\nu 
			\sum_{m\in\mathbb{Z}^{\flat}}m^2\hat{h}(m)\frac{e\big(m(x_1 + t + (\nu-2)\beta)\big)}{e(m\alpha)-1}.\notag
		\end{align}

		We claim that $\|m\alpha\|\geqslant1/(2|m|)$ for any $m\in \mathbb{Z}^\flat$. In fact, for any $m\in \mathbb{Z}^\flat$	there is a $q_k$ such that $q_k\leqslant m<q_{k+1}$.	Assume to the contrary that $\|m\alpha\|<1/(2|m|)$. 	Then there exits an	$s\in\mathbb{Z}$ such that $|m\alpha-s|<1/(2|m|)$. Therefore, we have
		\[
		\bigg|\alpha-\frac{s}{m}\bigg|<\frac{1}{2m^2}.
		\]
		Using Lemma \ref{con/fra/lem} (3), we deduce that $s/m=l_j/q_j$ for some $j>0$. 
		Then, using Lemma \ref{con/fra/lem} (4) we know that $(l_j,q_j)=1$, which implies $q_j|m$ and we can write $m=\xi q_j$ for some $\xi\in \mathbb{N}$. Since $|m|<q_{k+1}$, we have $j\leqslant k$.	But $q_k\nmid m$, hence $j<k$.		Then from the inequality
		\[
		\bigg|\alpha-\frac{l_j}{q_j}\bigg|=\bigg|\alpha-\frac{s}{m}\bigg|<\frac{1}{2m^2}
		\]
		we derive
		\[
		\xi|q_j\alpha-l_j|=\bigg|m\alpha-m\frac{l_j}{q_j}\bigg|
		=\big|m\alpha-s\big| 
		<\frac{1}{2|m|}\leqslant\frac{1}{2}.
		\]
		Hence $\xi\|q_j\alpha\|=\|m\alpha\|$.  Applying Lemma \ref{con/fra/lem} (2) 
		and noting that $j<k$ and $q_k\leqslant m<q_{k+1}$, we get
		\[
		\|m\alpha\|=\xi\|q_j\alpha\|>\frac{\xi}{2q_{j+1}}\geqslant\frac{1}{2q_{k}}\geqslant\frac{1}{2|m|}.
		\]
		This contradiction verifies our claim. 
		
		It follows from the claim that 
		\begin{align*}
			\max
			\bigg\{
			\bigg\| \sum_{\nu=2}^{\nu'} b_\nu g' \bigg\|_\infty, 
			\bigg\| \sum_{\nu=2}^{\nu'} b_\nu g'' \bigg\|_\infty \bigg\}
			\ll 
			4\pi^2\sum_{\nu=2}^{\nu'}|b_\nu| \sum_{m\in\mathbb{Z}^{\flat}} m^2e^{-\eta|m|} \cdot 2|m| 
			\ll  4\pi^2 \|b\|, 
			\notag
		\end{align*}
		and Lemma \ref{bd/coef/W} gives  
		\begin{align}\label{a}
			\sum_{s\in\mathbb{Z}}|a(s)|\ll (4\pi^2\|b\| +1)^2.
		\end{align}
		Similarly, for $F_k(t)$ with $k=2, \ldots, K$, we have
		\begin{align}
			& \bigg(\sum_{\nu=2}^{\nu'}b_\nu f_{\nu,k}(t)\bigg)'=2\pi i\sum_{\nu=2}^{\nu'}b_\nu\sum_{\substack{q_k\leqslant |m|<q_{k+1}\\q_k|m}}m\hat{h}(m)\frac{e(mx_1+m(\nu-2)\beta)}{e(m\alpha)-1}e(mt),\notag \\
			&\bigg(\sum_{\nu=2}^{\nu'}b_\nu f_{\nu,k}(t)\bigg)''=-4\pi^2 \sum_{\nu=2}^{\nu'}b_\nu\sum_{\substack{q_k\leqslant |m|<q_{k+1}\\q_k|m}}m^2\hat{h}(m)\frac{e(mx_1+m(\nu-2)\beta)}{e(m\alpha)-1}e(mt).\notag
		\end{align}
		Different from the case of $G(t)$, in the definition of $f_{\nu,k}$ we have $q_k|m$. Denote $m=aq_k$. Then $a<q_{k+1}/q_k$. 
		Using Lemma \ref{con/fra/lem} (2) we get
		\begin{align*}
			a\|q_k\alpha\|<\frac{a}{q_{k+1}}
			\leqslant
			\frac{q_{k+1}}{q_k}\frac{1}{q_{k+1}}
			\leqslant
			\frac{1}{q_{k}}<\frac{1}{2},
		\end{align*}
		which implies that 
		\begin{equation}\label{aqa}
			\|aq_k\alpha\|=a\|q_k\alpha\|.
		\end{equation}
		Then using $q_{k+1}\asymp e^{q_k}$, $\hat{h}(m)\ll e^{-\eta |m|}$ and Lemma \ref{con/fra/lem} (2) we get
		\begin{align}
			& \max\Big(\|\sum_{\nu=2}^{\nu'}b_\nu f_{\nu,k}'\parallel_\infty,\|\sum_{\nu=2}^{\nu'}b_\nu f''_k\|_\infty\Big)
			\ll 
			4\pi^2\sum_{\nu=2}^{\nu'}|b_\nu| 
			\sum_{1\leqslant a< q_{k+1}/q_k} (aq_k)^2 e^{-\eta aq_k}\|aq_k\alpha\|^{-1}\notag
			\\
			&\ll 
			4\pi^2 \|b\| \sum_{1\leqslant a<q_{k+1}/q_k}aq_k^{2}e^{-\eta aq_k}q_{k+1}\notag
			\ll 
			4\pi^2\|b\| q_k^{2}\bigg(1+  \sum_{2\leqslant a<q_{k+1}/q_k}
			\exp{(\log a-\eta aq_k+q_k)}\bigg)\notag
			\\
			&\ll 
			4\pi^2 \|b\| q_k^2.\notag
		\end{align}
		Therefore, by Lemma \ref{bd/coef/W}, for $k=2\ldots, K$ we have
		\begin{align}\label{b}
			\sum_{l\in\mathbb{Z}}|b^{(k)}{(l)}|\ll \big( 4\pi^2  \|b\| +1\big)^2q_k^4.
		\end{align}
		
		Finally, 
		noting that $q_{k+1} \asymp e^{q_k}$,
		there is a constant $C>1$ such that  $Cq_{k+1}\geqslant q_k^2$  
		for all $k=2, \ldots, K$. 
		Then $q_{K-1} \leqslant C^{1/2}q_K^{1/2}$ and $q_{K-2} \leqslant C^{1/2+1/4}q_K^{1/4}$, etc.
		Therefore,
		\begin{align}\label{K}
			q_2q_3\ldots q_K
			\leqslant 
			C^{\frac{1}{2}+\frac{3}{4}+\cdots}q_K^{1+\frac{1}{2}+\frac{1}{4}+\cdots}
			\leqslant 
			C^K q_K^2.  
		\end{align}
		Substituting \eqref{a}, \eqref{b} and \eqref{K} into \eqref{s1mn}, we get
		\begin{align}\label{s1mn2}
			S_1(M,N)
			&\ll_{D,\varepsilon} 
			(8(4\pi^2 \|b\| +1)^2)^K  \cdot(C^Kq_K^2)^{4}\cdot\frac{M}{\log^{D}N}\notag \\ 
			&\ll 
			\big(8C^4 (4\pi^2 \|b\| +1)^2 \big)^K  q_K^8 \frac{M}{\log^{D}N}.
		\end{align}

		Inserting \eqref{s1mn2} into \eqref{smn1}, we conclude that 
		\begin{align*}
			S(M,N)\ll \big(8C^4 (4\pi^2 \|b\| +1)^2 \big)^K  q_K^8 \frac{M}{\log^{D}N}.
		\end{align*}
		Now we need to estimate $K$. In fact, using  $q_{k+1} \asymp e^{q_k}$ and \eqref{KlogN} we get
		\[
		e^{q_{K-1}}  \ll   q_{K}\leqslant 2\log N,
		\]
		then
		\[
		K\leqslant q_{K-1}
		\ll
		\log\log N.
		\]
		Thus, we obtain 
		\[
		\big(8 C^4 (4\pi^2  \|b\| +1)^2 \big)^K
		\leqslant
		\log^{B}N
		\]
		for some constant $B>0$ depending on $b$ and $\alpha$. Therefore, we get
		\[
		S(M,N)
		\ll \frac{M}{\log^{D-B-8} N}
		\ll \frac{M}{\log^{A}N}
		\]
		by taking $D>A+B+8$. Then, $A> 0$ is arbitrary, and  \eqref{th1/goa} is proved with the implied constant depending  on $A$, $\varepsilon$ and $B$, hence on the trigonometric functions $e(\langle b,\cdot\rangle)$  and $\alpha$.
		
		\subsection{Proof of Theorem~\ref{thm1}:  conclusion} 
		From \eqref{th1/goa} we get \eqref{M/Rate} directly. Using the fact that trigonometric functions are dense in $C(\mathbb{T}^\omega)$, we see that \eqref{M/wo/Rate}  holds.

		\section{Proof of Theorem \ref{thm2}}
		\subsection{Proof of Theorem~\ref{thm2}: starting point}  
		Let $\mathcal{Q}$ be the set consisting of the denominators $\{q_k\}_{k=0}^\infty$ of 
		the convergents  of $\alpha$. Divide $\mathcal{Q}$  into two subsets 
		\begin{equation}\label{Qsp}
			\mathcal{Q}^{\sharp}:=\{q_k\in\mathcal{Q}:q_{k+1}>q_k^{\tau/3}, k\geqslant 1\}
		\end{equation} 
		and 
		$$
		\mathcal{Q}^{\flat}:=\{1\} \cup \{q_k\in\mathcal{Q}:q_{k+1}\leqslant q_k^{\tau/3}, k\geqslant1\}.
		$$
		Then we define
		$$
		M_1:= \bigcup_{q_k\in\mathcal{Q}^{\sharp}}\{m\in\mathbb{Z}: q_k\leqslant|m|<q_{k+1},q_k|m\},
		$$
		\begin{equation}\label{M2}
			M_2:=\bigcup\limits_{q_k\in \mathcal{Q}^{\flat}}\{m\in\mathbb{Z}: q_k\leqslant|m|<q_{k+1},q_k| m\},
		\end{equation} 
		and
		\begin{equation}\label{M3}
			M_3:= 
			\bigcup\limits_{q_k\in \mathcal{Q}}\{m\in\mathbb{Z}: q_k\leqslant|m|<q_{k+1},q_k\nmid m\}.		
		\end{equation} 
		It follows that $\mathbb{Z}=M_1\cup M_2\cup M_3\cup \{0\}$. Let $h: \mathbb{R} \to \mathbb{R}$ be $1$-periodic and $\tau$-smooth for some fixed $\tau>3$.  Denote the $m$-th Fourier coefficient of $h$ by $\hat{h}(m)$.  We call
		\begin{equation}\label{h'nu12}
			h_1'(t):=\sum_{m\in M_1}\hat{h}(m)e(mt), 
			\quad \quad
			h_2'(t):=
			\sum_{m\in M_2 \cup M_3}			\hat{h}(m)e(mt)
		\end{equation} 
		the resonant part and non-resonant part of $h$ respectively, where the dashes should not be confused with derivatives.   Thus,  
		\begin{align}\label{h=h1+h2}
			h(t)=\hat{h}(0)+h_1'(t)+h_2'(t).
		\end{align}

		Applying the following lemma of Wang and Yao in \cite[Lemma 2]{WanYao20} we can write  the non-resonant part $h_{2}'$ of $h$ as the difference of a  Lipschitz function.
		
		\begin{lemma}\label{lemma2}
			Let $h: \mathbb{R} \to \mathbb{R}$ be $1$-periodic and $\tau$-smooth for some fixed $\tau>3$ with Fourier coefficients $\hat{h}(m)$, $m\in\mathbb{Z}$.
			Then, the series
			\begin{align}
				\sum_{m\in {M_2\cup M_3}}
				\hat{h}(m)\frac{1}{e(m\alpha)-1}e(mt)\notag
			\end{align}
			converges uniformly to a Lipschitz function $\psi(t)$, where  $M_2$ and $M_3$ are defined in \eqref{M2} and \eqref{M3} respectively.
		\end{lemma}
		
		Using this lemma, we can write
		\begin{align}\label{h1=-}
			&h_{2}'(t)
			=\sum_{m \in {M_2\cup M_3}}\hat{h}(m)\frac{e(m\alpha)-1}{e(m\alpha)-1}e(mt) 
			=
			\psi(t+\alpha)-\psi(t)
		\end{align}
		for any $t\in [0, 1)$.

		Let $T_1: \mathbb{T}^{\omega}\to \mathbb{T}^{\omega}$ be defined by
		\begin{align}\label{t1}
			&T_1(x_1,x_2,\ldots, x_\nu,\ldots) \notag\\
			& \quad =(x_1+\alpha, x_2+\hat{h}(0)+h_{1}'(x_1),
			\ldots, x_\nu+\hat{h}(0)+h_{1}'(x_1+(\nu-2)\beta),\ldots). 
		\end{align} 
		Applying \eqref{t1} repeatedly, we get 
		\begin{align}\label{T1n}
			T_1^n (x_1, x_2, \ldots, x_\nu, \ldots) 
			& =(x_1'(n), x_2'(n), \ldots, x'_\nu(n), \ldots)  
		\end{align} 
		with $x'_1(n)=x_1+n\alpha$
		and, for $\nu\geqslant 2$, 
		\begin{align*}
			x_\nu'(n)=x_\nu+n\hat{h}(0)+\sum^{n-1}_{j=0}h_{1}'(x_1+j\alpha+(\nu-2)\beta). 
		\end{align*}
		Define $\Psi: \mathbb{T}^\omega\to \mathbb{T}^\omega$ by
		\begin{align*}
			\Psi(x_1, x_2, \ldots, x_\nu, \ldots)
			=(x_1, x_2-\psi(x_1), \ldots, x_\nu-\psi(x_1+(\nu-2)\beta), \ldots).
		\end{align*}
		Then $\Psi$ is an invertible Lipschitz function and 
		\begin{align*}
			\Psi^{-1}(x_1, x_2, \ldots, x_\nu, \ldots)
			=(x_1, x_2+\psi(x_1), \ldots, x_\nu+\psi(x_1+(\nu-2)\beta), \ldots).
		\end{align*} 
		Therefore, using \eqref{h=h1+h2} and \eqref{h1=-} we obtain 
		\begin{align*}
			&(\Psi^{-1}\circ T_1\circ \Psi)(x_1, x_2, \ldots, x_\nu, \ldots)\\
			& \quad =(\Psi^{-1}\circ T_1)(x_1, x_2-\psi(x_1), \ldots, x_\nu-\psi(x_1+(\nu-2)\beta), \ldots)\\
			& \quad =\Psi^{-1}(x_1+\alpha, x_2-\psi(x_1)+h_{1}'(x_1)+\hat{h}(0),\ldots, \\
			& \quad\quad\quad x_\nu-\psi(x_1+(\nu-2)\beta)+h_{1}'(x_1+(\nu-2)\beta)+\hat{h}(0), \ldots). 
		\end{align*}
		The last line is 
		\begin{align*}
			& =(x_1+\alpha,x_2+h(x_1), \ldots,x_\nu+h(x_1+(\nu-2)\beta),\ldots) \\ 
			& =T(x_1, x_2, \ldots). 
		\end{align*}
		Then, $T^n=\Psi^{-1}\circ T_1^n\circ \Psi$ for all $n\geqslant 1$. Therefore, for any continuous function
		$f$ on $\mathbb{T}^\omega$ and for any  $x\in \mathbb{T}^\omega$, we have 
		\begin{equation*}\label{fTn}
			f(T^n(x))=(f\circ \Psi^{-1})( T_1^n(\Psi(x))). 
		\end{equation*}
		Since  $\Psi$ is an invertible Lipschitz function, $f \circ \Psi^{-1}$ will run through all the continuous functions
		on $\mathbb{T}^\omega$ when $f$ runs through all the 
		continuous functions on $\mathbb{T}^\omega$. 
		Noting that $\Psi$ is invertible, $\Psi(x)$ runs through $\mathbb{{T}^\omega}$ when $x$ runs through $\mathbb{{T}^\omega}$. Hence, to prove \eqref{M/Rate}            for $T$, it suffices to prove \eqref{M/Rate} for $T_1$.
		
		Noting that the set of trigonometric functions is dense in $C(\mathbb{T^\omega})$,  to prove the theorem, it suffices to show that for any $x\in\mathbb{T}^\omega$ 	and $b \in\mathbb{Z}^\infty$,  we have
		\begin{equation}\label{S'MN}
			S'(M,N)
			:=\sum_{N-M<n\leqslant N}\mu(n)e(\langle b, T_1^n(x)\rangle)
			\ll \frac{M}{\log^{A} N}
		\end{equation} 
		for arbitrary $A>0$. 
		
		In the proof that follows, for any $b\in\mathbb{Z}^\infty$ we denote $b=(b_1,\ldots,b_{\nu'},0,\ldots)$ with $b_\nu\in\mathbb{Z}$ for $1\leqslant\nu\leqslant \nu'=\nu_b'$, $\rho(b):=\sum_{\nu=2}^{\nu'}b_\nu$ and $\|b\|:=\sum_{\nu=2}^{\nu'}|b_\nu|$.
		
		\subsection{Proof of Theorem~\ref{thm2}: proof of \eqref{S'MN}} 
		Applying \eqref{T1n}, we get
		\begin{align}\label{S'MN3}
			S'(M,N)
			& =e\bigg\{\sum_{\nu=1}^{\nu'}b_\nu x_\nu\bigg\}
			\sum_{N-M<n\leqslant N}\mu(n)
			e\bigg\{ b_1n\alpha+n \rho(b)  \hat{h}(0)\notag\\
			& \quad +\sum_{\nu=2}^{\nu'}b_\nu\sum_{j=0}^{n-1}h_{1}'(x_1+j\alpha+(\nu-2)\beta)\bigg\}. 
		\end{align} 
		For $\mathcal{Q}^{\sharp}$ defined in \eqref{Qsp}, write 
		\[
		\mathcal{Q}^{\sharp}=\{\tilde{q_1}, \tilde{q_2}, \ldots\}
		\]
		with $\tilde{q}_k<\tilde{q}_{k+1}$ and $\tilde{q}^+_{k}$  for the next term of $\tilde{q}_k$ in $\mathcal{Q}$. Then  we have
		\[
		2\leqslant \tilde{q}_2<\tilde{q}_2^{\tau/3}< \tilde{q}_3<\tilde{q}_3^{\tau/3}< \tilde{q}_4<\cdots .
		\]
		From the definition of $h_{1}'$ in \eqref{h'nu12}, we obtain
		\begin{align}
			\sum_{j=0}^{n-1}h_{1}'   (x_1+j\alpha +(\nu-2)\beta)
			&=\sum_{k\geqslant 1}
			\sum_{\substack{\tilde{q}_k\leqslant |m|<\tilde{q}_k^+\\\tilde{q}_k|m}}\hat{h}(m)e(m x_1+m(\nu-2)\beta)\frac{e(mn\alpha)-1}{e(m\alpha)-1}.\notag
		\end{align}
		We will split the above sum at $k=K'$ for some proper $K'$ to be chosen. 
		Let $K'=K'_N$ be the positive integer satisfying 
		$$
		\tilde{q}_{K'}<N\leqslant\tilde{q}_{K'}^+,
		$$
		which implies that
		\[
		2^{(\tau /3)^{K'-2}}\leqslant \tilde{q}_2^{(\tau /3)^{K'-2}}<\tilde{q}_{K'}<N.
		\]
		Then, we have $K'\ll \log \log N$. Since $h$ is ${\tau}$-smooth,  
		for $n\leqslant N$ we have  
		\begin{align}\label{SK'}
			S_{\nu,K'}(n):=& \sum_{k\geqslant K'}
			\sum_{\substack{\tilde{q}_k\leqslant |m|<\tilde{q}_k^+\\\tilde{q}_k|m}}\hat{h}(m)e(m x_1+m(\nu-2)\beta)\frac{e(mn\alpha)-1}{e(m\alpha)-1}
			\\
			\ll&  
			\sum_{m\geqslant \tilde{q}_{K'}}|m|^{-\tau}N
			\ll
			N^{2-\tau}.\notag
		\end{align} 
		Noticing that  \eqref{S'MN3} can be written as
		\begin{align*}
			&S'(M,N)
			\\
			&\hskip-1mm \ll \hskip-1mm  
			\sum_{N-M<n\leqslant N}\mu(n)e\hskip-0.5mm
			\bigg\{b_1n\alpha + n \rho(b)  \hat{h}(0)
			\hskip-0.5mm
			+
			\hskip-0.5mm
			\sum_{\nu=2}^{\nu'}b_\nu	\bigg( \sum_{1\leqslant k\leqslant K'}\eta_{\nu,k}(\tilde{q}_kn\alpha)+S_{\nu,K'}(n)\bigg)\bigg\}
			\\
			&\hskip-1mm \ll \hskip-1mm  
			\sum_{N-M<n\leqslant N}\mu(n)e\bigg\{b_1n\alpha+  n \rho(b) \hat{h}(0) + 
			\sum_{\nu=2}^{\nu'}b_\nu 
			\sum_{1\leqslant k\leqslant K'}\eta_{\nu,k}(\tilde{q}_kn\alpha)\bigg\} 
			\cdot
			e\bigg\{\sum_{\nu=2}^{\nu'}b_\nu S_{\nu,K'}(n) \bigg\},
		\end{align*} 
		where
		\begin{align}\label{etak}
			\eta_{\nu,k}(t):=\sum_{r=1}^{\tilde{q}_k^+/\tilde{q}_k}\hat{h}(r\tilde{q}_k)e(r\tilde{q}_kx_1+r\tilde{q}_k(\nu-2)\beta)\frac{e(rt)-1}{e(r\tilde{q}_k\alpha)-1},
		\end{align}
		applying Taylor's formula to the last factor $e\big(\sum_{\nu=2}^{\nu'}b_\nu   S_{\nu,K'}(n)\big)$, using \eqref{SK'} and noting that $\tau>3$, we get		
		\begin{equation}\label{S'MN2}
			S'(M,N)
			\ll  S_1'(M,N)\big( 1+ \|b\|   N^{2-\tau}   \big)
			\ll S_1'(M,N)+
			\|b\|,
		\end{equation}
		where
		\begin{equation} \label{DS1}
			S_1'(M,N):=\sum_{N-M<n\leqslant N}\mu(n)e
			\Big\{ b_1n\alpha+n \rho(b)  \hat{h}(0) +\sum_{1\leqslant k \leqslant K'}\sum_{\nu=2}^{\nu'}b_\nu\eta_{\nu,k}(\tilde{q}_kn\alpha) \Big\}.
		\end{equation} 
		So, to get \eqref{S'MN},  we first need to estimate $S_1'(M,N)$.  
		
		For $1\leqslant k \leqslant K'$, let
		\[
		H_k(t):=e\Big(\sum_{\nu=2}^{\nu'}b_\nu\eta_{\nu,k}(t)\Big).
		\]
		Then $H_k(x)$ is an analytic $1$-periodic function. So we can expand it into the Fourier series
		\[
		H_k(t)=\sum_{l_k\in\mathbb{Z}}c^{(k)}{(l_k)}e(l_kt).
		\]
		Substituting them into \eqref{DS1}, we get
		\begin{align*}
			&S_1'(M,N) \\ 
			&=	\sum_{N-M<n\leqslant N}\mu(n) e\{b_1n\alpha+n \rho(b)  \hat{h}(0)\}
			\prod_{k=1}^{K'}  H_k(\tilde{q}_kn\alpha)\notag\\
			&= \sum_{l_1\in\mathbb{Z}}c^{(1)}{(l_1)}
			\cdots
			\sum_{l_{K'}\in\mathbb{Z}}c^{(K')}{(l_{K'})}
			\sum_{N-M<n\leqslant N}\mu(n)
			e\Big\{n\alpha
			\Big(b_1 + \sum_{k=1}^{K'}l_k\tilde{q}_k\Big)
			+
			n\rho(b)  \hat{h}(0) 
			\Big\}\notag\\
			&\ll\sum_{l_1\in\mathbb{Z}}c^{(1)}{(l_1)}\cdots\sum_{l_{K'}\in\mathbb{Z}}c^{(K')}{(l_{K'})}\notag\\
			&\quad \times \sup_{l_1,\ldots, l_{K'}\in \mathbb{Z}}\sum_{N-M<n\leqslant N}\mu(n)
			e\Big\{n\alpha
			\Big(b_1 + \sum_{k=1}^{K'}l_k\tilde{q}_k\Big)
			+
			n\rho(b) \hat{h}(0) 
			\Big\}. 
		\end{align*}
		Now by Lemma \ref{zhan} the last line is 
		\begin{align*}
			\ll \frac{M}{ \log^{D'} N}  
		\end{align*}
		for any $N^{5/8+\varepsilon} \leqslant M \leqslant N$ and $D'>0$, from which it follows that 
		\begin{align} \label{S1'MN}
			S_1'(M,N)
			\ll 
			\frac{M}{\log^{D'} N}
			\sum_{l_1\in\mathbb{Z}} |c^{(1)}{(l_1)}| \cdots\sum_{l_{K'}\in\mathbb{Z}} |c^{(K')}{(l_{K'})}|.  
		\end{align}
		So, to get \eqref{S'MN}, it suffice to estimate each summation over $l_k\in\mathbb{Z}$ for $k=1, \ldots, K'$.

		Now, we will estimate each summation over $l_k\in\mathbb{Z}$ for $k=1, \ldots, K'$. 
		To use Lemma \ref{bd/coef/W}, we need to calculate the uniform norm of $\eta_{\nu,k}'$ and $\eta_{\nu,k}''$.  
		From the definition of $\eta_{\nu,k}$ in \eqref{etak}, we get
		\begin{align}
			& \bigg\{\sum_{\nu=2}^{\nu'}b_\nu \eta_{\nu,k}(t)\bigg\}'=2\pi i\sum_{\nu=2}^{\nu'}b_\nu\sum_{r=1}^{\tilde{q}_k^+/\tilde{q}_k}r\hat{h}(r\tilde{q}_k)e(r\tilde{q}_kx_1+r\tilde{q}_k(\nu-2)\beta)\frac{e(rt)}{e(r\tilde{q}_k\alpha)-1},\notag \\
			&\bigg\{\sum_{\nu=2}^{\nu'}b_\nu \eta_{\nu,k}(t)\bigg\}''
			=
			-4\pi^2 \sum_{\nu=2}^{\nu'}b_\nu \sum_{r=1}^{\tilde{q}_k^+/\tilde{q}_k}r^2\hat{h}(r\tilde{q}_k)e(r\tilde{q}_kx_1+r\tilde{q}_k(\nu-2)\beta)\frac{e(rt)}{e(r\tilde{q}_k\alpha)-1},\notag
		\end{align}
		where $\tilde{q}_k\in\mathcal{Q}^{\sharp}$. 		
		Noticing that here we have $1\leqslant r\leqslant \tilde{q}_{k}^+/\tilde{q}_k$, similar to \eqref{aqa}, we can also use Lemma \ref{con/fra/lem} (2) to get
		\begin{align*}
			r\|\tilde{q}_k\alpha\|
			<\frac{r}{\tilde{q}_{k}^+}
			\leqslant
			\frac{\tilde{q}_{k}^+}{\tilde{q}_k}\frac{1}{\tilde{q}_{k}^+}
			\leqslant
			\frac{1}{\tilde{q}_{k}}<\frac{1}{2},
		\end{align*}
		which implies that 
		$\|r\tilde{q}_k\alpha\|=r\|\tilde{q}_k\alpha\|$.
		Then
		\begin{align}
			\max\Big(\|
			\big(\sum_{\nu=2}^{\nu'}b_\nu\eta_{\nu,k}\big)'\|_{\infty},\|\big(\sum_{\nu=2}^{\nu'}b_\nu\eta_{\nu,k}\big)''\|_{\infty}\Big)
			&\ll
			4\pi^2 \sum_{\nu=2}^{\nu'} |b_\nu| 
			\sum_{r=1}^{\tilde{q}_k^+/\tilde{q}_k}r^2|\hat{h}(r\tilde{q}_k)|\|r\tilde{q}_k\alpha\|^{-1}\notag\\
			&\ll
			4\pi^2 \|b\| 
			\tilde{q}_k^+            \sum_{r=1}^{\tilde{q}_k^+/\tilde{q}_k}|r  \hat{h}(r\tilde{q}_k)|.\notag 
		\end{align}
		Applying Lemma \ref{bd/coef/W}, we get that
		\[
		\sum_{l_k\in\mathbb{Z}}|c^{(k)}{(l_k)}|
		\ll
		4\cdot 2 \cdot (4\pi^2)^2\bigg( \|b\|  \tilde{q}_k^+\sum_{r=1}^{\tilde{q}_k^+/\tilde{q}_k}|r  \hat{h}(r\tilde{q}_k)|+1\bigg)^2,
		\]
		and then
		\begin{align}	\label{pck}
			\sum_{l_1\in\mathbb{Z}}c^{(1)}{(l_1)}\cdots\sum_{l_{K'}\in\mathbb{Z}}  c^{(k)}{(l_k)}
			&\ll
			\prod_{k=1}^{K'} 2 (8\pi^2)^2  
			\bigg( \|b\| \tilde{q}_k^+\sum_{r=1}^{\tilde{q}_k^+/\tilde{q}_k}|r    \hat{h}(r\tilde{q}_k)|+1\bigg)^2  \notag
			\\
			& 
			\ll
			2^{K'}(8\pi^2)^{2K'}\prod_{k=1}^{K'}
			\bigg(
			\|b\|  \tilde{q}_k^+
			\tilde{q}_k^{-\tau}+1
			\bigg)^2
		\end{align}
		by noting that $\tau > 3$ and the $\tau$-smoothness of $h$ implies $\hat{h}(n) \ll n^{-\tau}$.	
		Since $\tilde{q}_{k}^+ \ll \tilde{q}_k^\tau$, there is a constant
		$C_k > 1$ such that $ \tilde{q}_k^+ 	\tilde{q}_k^{-\tau}\leqslant C_k$. Then 
		\begin{align*}
			\|b\| \tilde{q}_k^+
			\tilde{q}_k^{-\tau}+1
			\leqslant \|b\|C_k+1.
		\end{align*}
		Substituting this into \eqref{pck}, writing $C'=\max\{C_k,1 \leqslant k\leqslant K'\}$, and using 
		$K'\leqslant\log\log N$, we have $C'>1$, $\|b\| \ll \log^{B'} N$ and
		\[
		\sum_{l_1\in\mathbb{Z}}c^{(1)}{(l_1)}\cdots\sum_{l_{K'}\in\mathbb{Z}}  
		c^{(k)}{(l_k)}
		\ll
		2^{K'}(8\pi^2 ( \|b\|C' +1))^{2K'}
		\ll
		\log^{B'} N
		\]
		for some constant $B'>0$ depending on $ b $ and $C'$, thus depending on the trigonometric functions $e(\langle b,\cdot\rangle)$ and $\alpha$. Combining this with \eqref{S1'MN}  and 	taking $D'>B'$, $A = D'-B'$, we see that $A>0$ is arbrtrary and 
		\begin{equation*}\label{S'1MNbound}
			S'_1(M,N)\ll \frac{M}{\log^{D'-B'} N}\ll \frac{M}{\log^{A}N}
		\end{equation*} 
		with the implied constant depending on $A,\alpha, \varepsilon$ and $e(\langle b,\cdot\rangle)$.

		Therefore, \eqref{S'MN2} becomes
		\begin{equation*} 
			S'(M,N) \ll \frac{M}{\log^{A}N}
			+
			\log^{B'} N
			\ll \frac{M}{\log^{A}N}
		\end{equation*}    
		for any  $A>0$, i.e., \eqref{S'MN} is true.

		\subsection{Proof of Theorem~\ref{thm2}:  conclusion}
		Combining the fact that trigonometric functions are dense in $C(\mathbb{T}^\omega)$ with   \eqref{S'MN}, we see that \eqref{M/wo/Rate}  holds for any $f \in C(\mathbb{T}^\omega)$.

		\section{Proof of Theorem \ref{thm3}}
		If $\alpha=0$, then \eqref{smn} becomes
		\[
		S(M,N)
		\ll
		\sum_{N-M<n\leqslant N}
		\mu(n)e\bigg\{n\sum_{\nu =2}^{\nu'}b_\nu h(x_1+(\nu-2)\beta)\bigg\}
		\ll \frac{M}{\log^A N}
		\]
		for any $A>0$, where in the last step we have used Zhan's result in Lemma \ref{zhan}.
		
		Suppose that $0\neq\alpha=l/q$ with $(l, q)=1$.  Divide the last sum over $j$ in \eqref{smn} into residue classes modulo $q$. 
		Since $h$ is $1$-periodic,  
		letting $n\equiv r\bmod{q}$ for some $0\leqslant r\leqslant q-1$, 
		we can write 
		\[
		\sum_{j=0}^{n-1}h(x_1+j\alpha+(\nu-2)\beta)
		=\gamma_{1}
		\bigg(\frac{n-r}{q}+1\bigg)+\gamma_{2}\frac{n-r}{q},
		\]
		where
		\[
		\gamma_{1}=\sum_{j=0}^{r-1}h
		\bigg(x_1+\frac{jl}{q}+(\nu-2)\beta\bigg), 
		\quad 
		\gamma_{2}=\sum_{j=r}^{q-1}h\bigg(x_1+\frac{jl}{q}+(\nu-2)\beta\bigg).
		\]
		Applying Lemma \ref{zhan} again, we get
		\begin{align*}\label{rational}
			S(M, N)
			&\ll \sum_{r=0}^{q-1}\sum_{\substack{N-M<n\leqslant N\\ n\equiv r \bmod q}} 
			\mu(n)e\bigg\{b_1n\alpha+\sum_{\nu=2}^{\nu'} b_{\nu}\bigg(\gamma_{1} \bigg(\frac{n-r}{q}+1\bigg)
			+\gamma_{2}\frac{n-r}{q}\bigg)\bigg\}\notag\\
			&\ll_{A,q,\varepsilon} \frac{M}{\log^A N}
		\end{align*}
		for any $A>0$. 
		This proves Theorem \ref{thm3}.

		\subsection*{Funding} Liu is supported by the National Natural Science 
		Foundation of China (Grant No. 12401012).

	\end{document}